\documentclass[11pt]{amsart}
\newtheorem{theorem}{Theorem}

\newtheorem{corollary}[theorem]{Corollary}

\newtheorem{definition}[theorem]{Definition}
\newtheorem{lemma}[theorem]{Lemma}

\newtheorem{proposition}[theorem]{Proposition}

\usepackage{color}
\usepackage[colorlinks, bookmarks=true]{hyperref}
\usepackage{color,graphicx,shortvrb}
\usepackage[active]{srcltx} 

\usepackage[latin 1]{inputenc}

\usepackage[active]{srcltx} 

\def\J#1#2#3{ \left\{ #1,#2,#3 \right\} }

\def\NN{{\mathbb{N}}}
\def\11{\textbf{$1$}}
\def\CC{{\mathbb{C}}}

\usepackage{enumerate}
\usepackage{amsmath}
\usepackage{amssymb}

\begin{document}
\title[O.P. \& O.A. holomorphic mappings]{Orthogonally additive, orthogonality preserving, holomorphic mappings between C$^*$-algebras}

\author[J.J. Garc{\' e}s]{Jorge J. Garc{\' e}s}
\email{jgarces@ugr.es}
\address{Departamento de An{\'a}lisis Matem{\'a}tico, Facultad de
Ciencias, Universidad de Granada, 18071 Granada, Spain.}

\author{Antonio M. Peralta}

\email{aperalta@ugr.es}
\address{Departamento de An{\'a}lisis Matem{\'a}tico, Facultad de
Ciencias, Universidad de Granada, 18071 Granada, Spain.}

\author{Daniele Puglisi}

\email{dpuglisi@dmi.unict.it}

\address{Department of Mathematics and Computer Sciences, University of
Catania, Catania, 95125, Italy}

\thanks{Authors partially partially supported by the Spanish Ministry of Economy and Competitiveness,
D.G.I. project no. MTM2011-23843, and Junta de Andaluc\'{\i}a grant FQM3737.}

\author[Ram\'{i}rez]{Mar{\'\i}a Isabel Ram{\'\i}rez}
\address{Departamento de Algebra y An\'alisis Matem\'atico, Universidad de
Almer\'ia, 04120 Almer\'ia, Spain} \email{mramirez@ual.es}

\date{}
\maketitle

\begin{abstract} We study holomorphic maps between C$^*$-algebras $A$ and $B$. When $f:B_A (0,\varrho) \longrightarrow B$ is a holomorphic mapping whose Taylor series at zero is uniformly converging in some open unit ball $U=B_{A}(0,\delta)$ and we assume that $f$ is orthogonality preserving on $A_{sa}\cap U$, orthogonally additive on $U$ and $f(U)$ contains an invertible element in $B$, then there exist a sequence $(h_n)$ in $B^{**}$ and Jordan $^*$-homomorphisms $\Theta, \widetilde{\Theta} : M(A) \to B^{**}$ such that $$ f(x) = \sum_{n=1}^\infty h_n \widetilde{\Theta} (a^n)= \sum_{n=1}^\infty {\Theta} (a^n) h_n,$$ uniformly in $a\in U$. When $B$ is abelian the hypothesis of $B$ being unital and $f(U)\cap \hbox{inv} (B) \neq \emptyset$ can be relaxed to get the same statement.
\end{abstract}

\medskip
\noindent
{\bf 2010 MSC: }Primary 46G20, 46L05; Secondary 46L51, 46E15, 46E50.\\
{\bf Keywords and phrases:} C$^*$-algebra, von Neumann algebra,
orthogonally\hyphenation{Ortho-gonally} additive holomorphic
functions, orthogonality preservers, orthomorphism, non-commutative $L_1$-spaces.

\maketitle

\section{Introduction}

The description of orthogonally additive $n$-homogeneous polynomial on $C(K)$-spaces and on general C$^*$-algebras, developed by Y. Benyamini, S. Lassalle, J.L.G. Llavona \cite{BLL} and D. P{\'e}rez, and I. Villanueva \cite{PerezVi} and C. Palazuelos, A.M. Peralta and I. Villanueva \cite{PPV}, respectively (see also \cite{CLZ06} and \cite[\S 3]{BurFerGarPe}), led Functional Analysts to study and explore orthogonally additive holomorphic functions on $C(K)$-spaces (see \cite{CLZ,JarPriZal}) and subsequently on general C$^*$-algebras (cf. \cite{PePugl}).\smallskip

We recall that a mapping $f$ from a C$^*$-algebra $A$ into a Banach space $B$ is said to be \emph{orthogonally additive} on a subset $U\subseteq A$ if for every $a,b$ in $U$ with $a \perp
b$, and $a+b\in U$ we have $f(a+b) = f(a) +f(b)$, where elements $a,$ $b$ in $A$ are said to be
\emph{orthogonal} (denoted by $a \perp b$) whenever $a b^* = b^* a=0$. We shall say that $f$ is \emph{additive on elements having zero-product} if for every $a,b$ in $A$ with $a b = 0 $ we have $f(a+b) = f(a) +f(b)$. 
Having this terminology in mind, the description of all $n$-homogeneous polynomials on a general C$^*$-algebra, $A,$ which are orthogonally additive on the self adjoint part, $A_{sa}$, of $A$ reads as follows (see section \S 2 for concrete definitions not explained here).

\begin{theorem}\label{thm PaPeVill}\cite{PPV}
Let $A$ be a C$^*$-algebra, $B$ a Banach space, $n\in \mathbb N,$ and let
$P:A \to B$ be an $n$-homogeneous polynomial. The following statements are
equivalent:
\begin{enumerate}[$(a)$]
\item There exists a bounded linear operator $T:A\to X$ satisfying
$$P(a)=T (a^n),$$ for every $a\in A,$ and $\|P\| \leq \|T\| \leq 2 \|P\|$.
\item $P$ is additive on elements having zero-products. \item $P$
is orthogonally additive on $A_{sa}$.$\hfill\Box$
\end{enumerate}
\end{theorem}

The task of replacing $n$-homogeneous polynomials by polynomials or by holomorphic functions involves a higher difficulty. For example, as noticed by D. Carando, S. Lassalle and I. Zalduendo \cite[Example 2.2.]{CLZ}, when $K$ denotes the closed unit disc in $\mathbb{C}$, there is no entire function $\Phi : \CC \to \CC$ such that the mapping $h: C(K) \to C(K)$, $h(f) = \Phi \circ f$ factors all degree-2 orthogonally additive scalar polynomials over $C(K)$. Furthermore, similar arguments show that, defining $P:C([0,1])\to \mathbb{C}$, $P(f) = f(0) + f(1)^2$, we cannot find a triplet $(\Phi,\alpha_1,\alpha_2)$, where  $\Phi: C[0,1] \to \mathbb{C}$ is a $^*$-homomorphism and $\alpha_1,\alpha_2 \in \mathbb{C}$, satisfying that $P(f) = \alpha_1 \Phi (f) + \alpha_2 \Phi (f^2)$ for every $f\in C([0,1])$.\smallskip

To avoid the difficulties commented above, Carando, Lassalle and Zalduendo introduce a factorization through an $L_1 (\mu)$ space. More concretely, for each compact Hausdorff space $K$, a holomorphic mapping of bounded type $f : C(K) \to \mathbb{C}$ is orthogonally additive if and only if there exist a Borel regular measure $\mu$ on $K$, a sequence $(g_k)_k \subseteq L_1(\mu)$ and a holomorphic function of bounded type $h: C(K) \to L_1(\mu)$ such that $\displaystyle h(a) = \sum_{k=0}^\infty g_k~ a^k,$ and
$$f(a) = \int_K h(a)~d\mu,$$  for every  $a\in C(K)$ (cf. \cite[Theorem 3.3]{CLZ}).\smallskip

When $C(K)$ is replaced with a general C$^*$-algebra $A$, a holomorphic function of bounded type $f: A\to \mathbb{C}$ is orthogonally additive on
$A_{sa}$ if and only if there exist a positive functional $\varphi$ in $A^*$, a sequence $(\psi_n)$ in $L_1
(A^{**},\varphi)$ and a power series holomorphic function $h$ in $\mathcal{H}_b(A, A^*)$ such that
$$h(a) = \sum_{k=1}^{\infty} \psi_k \cdot a^k \hbox{ and }
f(a) = \langle 1_{_{A^{**}}},  h(a) \rangle = \int h(a) \
d\varphi,$$ for every $a$ in $A$, where $1_{_{A^{**}}}$ denotes
the unit element in $A^{**}$ and $L_1 (A^{**},\varphi)$ is a
non-commutative $L_1$-space (cf. \cite{PePugl}).\smallskip

A very recent contribution due to Q. Bu, M.-H. Hsu, and N.-Ch. Wong \cite{BuHsuWong2013}, shows that, for holomorphic mappings between $C(K)$, we can avoid the factorization through an $L_1 (\mu)$-space by imposing additional hypothesis. Before stating the detailed result, we shall set down some definitions.\smallskip

Let $A$ and $B$ be C$^*$-algebras. When $f:U\subseteq A\to B$ is a map and the condition \begin{equation}
 \label{eq OP maps} a\perp b \Rightarrow f(a)\perp f(b)
 \end{equation}(respectively, \begin{equation}\label{eq zero product preserver} a b =0 \Rightarrow f(a) f(b) =0\ )
 \end{equation} holds for every $a,b\in U$, we shall say that $f$ \emph{preserves orthogonality} or is \emph{orthogonality preserving} (respectively, $f$ \emph{preserves zero products}) on $U$.
In the case $A=U$ we shall simply say that $f$ is \emph{orthogonality preserving}  (respectively, $f$ \emph{preserves zero products}). Orthogonality preserving bounded linear maps between C$^*$-algebras were completely described in \cite[Theorem 17]{BurFerGarMarPe} (see \cite{BurFerGarPe} for completeness).\smallskip

The following Banach-Stone type theorem for zero product preserving or orthogonality preserving holomorphic functions between $C_0(L)$ spaces is established by Bu, Hsu and Wong in \cite[Theorem 3.4]{BuHsuWong2013}.

\begin{theorem}\label{t BuShuWong}\cite{BuHsuWong2013} Let $L_1$ and $L_2$ be locally compact Hausdorff spaces and let $H: B_{C_0(L_1)} (0, r) \to  C_0(L_2 )$ be a bounded orthogonally additive holomorphic function. If $H$ is zero product preserving or orthogonality preserving, then there exist a sequence $(\mathcal{O}_n)$ of open subsets of $L_2$, a sequence $(h_n)$ of bounded functions from $L_2\cup \{\infty\}$ into $\mathbb{C}$ and a mapping $\varphi : L_2 \to L_1$ such that for each natural $n$ the function $h_n$ is continuous and nonvanishing on $\mathcal{O}_n$ and $$ f(a) (t) = \sum_{n=1}^\infty h_n (t) \left(a(\varphi(t))\right)^n, (t\in L_2),$$ uniformly in $a\in B_{C_0(L_1)}(0,r)$.$\hfill\Box$\end{theorem}

The study developed by Bu, Hsu and Wong restricts to commutative C$^*$-algebras or to orthogonality preserving and orthogonally additive, $n$-homogeneous polynomials between general C$^*$-algebras. The aim of this paper is to extend their study to holomorphic maps between general C$^*$-algebras. In Section 4, we determine the form of every orthogonality preserving, orthogonally additive holomorphic function from a general C$^*$-algebra into a commutative C$^*$-algebra (see Theorem \ref{thm OP + OA Holom commutative}).\smallskip

In the wider setting of holomorphic mappings between general C$^*$-algebras, we prove the following: Let $A$ and $B$ be C$^*$-algebras with $B$ unital and let $f:B_A (0,\varrho) \longrightarrow B$ be a holomorphic mapping whose Taylor series at zero is uniformly converging in some open unit ball $U=B_{A}(0,\delta)$. Suppose $f$ is orthogonality preserving on $A_{sa}\cap U$, orthogonally additive on $U$ and $f(U)$ contains an invertible element. Then there exist a sequence $(h_n)$ in $B^{**}$ and Jordan $^*$-homomorphisms $\Theta, \widetilde{\Theta} : M(A) \to B^{**}$ such that $$ f(x) = \sum_{n=1}^\infty h_n \widetilde{\Theta} (a^n)= \sum_{n=1}^\infty {\Theta} (a^n) h_n,$$ uniformly in $a\in U$ (see Theorem \ref{thm OP + OA Holom general}).\smallskip

The main tool to establish our main results is a newfangled investigation on orthogonality preserving pairs of operators between C$^*$-algebras developed in Section 3. Among the novelties presented in Section 3, we find an innovating alternative characterization of orthogonality preserving operators between C$^*$-algebras which complements the original one established in \cite{BurFerGarMarPe} (see Proposition \ref{c pairs OP and OP operators}). Orthogonality preserving pairs of operators are also valid to determine orthogonality preserving operators and orthomorphisms or local operators on C$^*$-algebras in the sense employed by A.C. Zaanen \cite{Za} and B.E. Johnson \cite{John01}, respectively.\smallskip

\section{Orthogonally additive, orthogonality preserving, holomorphic mappings on C$^*$-algebras}

Let $X$ and $Y$ be Banach spaces. Given a natural $n$, a (continuous) $n$-homogeneous polynomial $P$
from $X$ to $Y$ is a mapping $P: X \longrightarrow Y$ for which
there is a (continuous) multilinear symmetric operator
$A: X\times \ldots \times X \to Y$ such that
$P(x) = A(x , \ldots , x), \ \text{for every}\  x \in X.$
All the polynomials considered in this paper are assumed to be continuous. By a $0$-homogeneous polynomial we mean a constant function. The symbol $\mathcal{P}(^n X , Y)$ will denote the Banach space of all
continuous $n$-homogeneous polynomials from $X$ to $Y$, with norm given by
$\displaystyle \|P\| = \sup_{\|x\|\leq 1} \|P(x)\|.$
\smallskip

Throughout the paper, the word operator will always stand for a bounded linear mapping.\smallskip

We recall that, given a domain $U$ in a complex Banach space $X$ (i.e. an open, connected subset), a function $f$ from $U$ to another complex Banach space $Y$ is said to be \emph{holomorphic} if the Fréchet derivative of $f$ at $z_0$ exists for every point $z_0$ in $U$. It is known that $f$ is holomorphic in $U$ if and only if for each $z_0 \in X$ there exists a sequence $\left(P_k(z_0)\right)_k$ of polynomials from $X$ into $Y$, where each $P_k (z_0)$ is $k$-homogeneous, and a neighborhood $V_{z_0}$ of $z_0$ such that the series $$ \sum_{k=0}^\infty P_k(z_0)  (y - z_0) $$
converges uniformly to $f(y)$ for every $y \in V_{z_0}$. Homogeneous polynomials on a C$^*$-algebra $A$ constitute the most basic examples of holomorphic functions on $A$. A holomorphic function $f :X\longrightarrow Y$ is said to be of {\em bounded type} if it is bounded on all bounded subsets of $X$,
in this case its Taylor series at zero, $f =
\sum_{k=0}^\infty P_k,$ has infinite radius of uniform convergence,
i.e.  $\limsup_{k \rightarrow \infty}
\|P_k\|^{\frac{1}{k}} = 0$ (compare \cite[\S
6.2]{Dineen}, see also \cite{Gam}).\smallskip

Suppose $f: B_X(0,\delta) \to Y$ is a holomorphic function and let $\displaystyle f=
\sum_{k=0}^\infty P_k$ be its Taylor series at zero which is assumed to be uniformly convergent in $U=B_X(0,\delta)$. Given $\varphi \in Y^*$, it follows from Cauchy's integral formula that, for each $a\in U$, we have: $$\varphi P_n (a) = \frac{1}{2 \pi i} \int_{\gamma} \frac{\varphi f (\lambda a)}{\lambda^{n+1}} d\lambda,$$ where $\gamma$ is the circle forming the boundary of a disc in the complex plane $D_{\mathbb{C}} (0,r_1),$ taken counter-clockwise, such that $a+ D_{\mathbb{C}} (0,r_1) a \subseteq U$. We refer to \cite{Dineen} for the basic facts and definitions used in this paper.\smallskip

In this section we shall study orthogonally
additive, orthogonality preserving, holomorphic mappings between C$^*$-algebras. We begin with an observation which can be directly derived from Cauchy's integral formula. The statement in the next lemma was originally stated by D. Carando, S. Lassalle and I. Zalduendo in \cite[Lemma 1.1]{CLZ} (see also \cite[Lemma 3]{PePugl}).

\begin{lemma}\label{l 1.1 in CarLassZal}
Let $f:B_A (0,\varrho) \longrightarrow B$ be a holomorphic mapping, where $A$ is a C$^*$-algebra and $B$ is a complex Banach space, and let $\displaystyle f=
\sum_{k=0}^\infty P_k$ be its Taylor series at zero, which is uniformly converging in $U=B_A (0,\delta)$. Then the mapping $f$ is orthogonally additive on $U$ {\rm(}respectively, orthogonally additive on
$A_{sa}\cap U$ or additive on elements having zero-product in $U${\rm)} if, and only
if, all the $P_k$'s satisfy the same property. In such a case,
$P_0 =0$.$\hfill\Box$
\end{lemma}

We recall that a functional $\varphi$ in the dual of a C$^*$-algebra $A$ is \emph{symmetric} when $\varphi (a)\in \mathbb{R}$, for every $a\in A_{sa}$. Reciprocally, if $\varphi (b) \in \mathbb{R}$ for every symmetric functional $\varphi \in A^*$, the element $b$ lies in $A_{sa}$. Having this in mind, our next lemma also is a direct consequence of the Cauchy's integral formula. A mapping $f: A\to B$ between C$^*$-algebras is called \emph{symmetric} whenever $f(A_{sa})\subseteq B_{sa}$, or equivalently, $f(a) = f(a)^*$, whenever $a\in A_{sa}$.

\begin{lemma}\label{l symmetric hol functions}
Let  $f:B_A (0,\varrho) \longrightarrow B$ be a holomorphic mapping, where $A$ and $B$ are C$^*$-algebras, and let $\displaystyle f=
\sum_{k=0}^\infty P_k$ be its Taylor series at zero, which is uniformly converging in $U=B_A (0,\delta)$.
Then the mapping $f$ is symmetric on $U$ {\rm(}i.e. $f(A_{sa}\cap U) \subseteq B_{sa}${\rm)} if, and only
if, $P_k$ is symmetric {\rm(}i.e. $P_k (A_{sa}) \subseteq B_{sa}${\rm)} for every $k\in \mathbb{N}\cup \{0\}$.$\hfill\Box$
\end{lemma}

\begin{definition}\label{def pair OP} Let $S,T: A\to B$ be a couple of mappings between two C$^*$-algebras. We shall say that the pair $(S,T)$ is orthogonality preserving on a subset $U\subseteq A$ if $S(a)\perp T(b)$ whenever $a\perp b$ in $U$. When $a b=0$ in $U$ implies $S(a) T(b)=0$ in $B$, we shall say that $(S,T)$ preserves zero products on $U$.
\end{definition}

We observe that a mapping $T: A\to B$ is orthogonality preserving in the usual sense if and only if the pair $(T,T)$ is orthogonality preserving. We also notice that $(S,T)$ is orthogonality preserving (on $A_{sa}$) if and only if $(T,S)$ is orthogonality preserving (on $A_{sa}$).\smallskip

Our next result assures that the $n$-homogeneous polynomials appearing in the Taylor series of an orthogonality preserving holomorphic mapping between C$^*$-algebras are pairwise orthogonality preserving.

\begin{proposition}\label{p OP holom mappings}
Let  $f:B_A (0,\varrho) \longrightarrow B$ be a holomorphic mapping, where $A$ and $B$ are C$^*$-algebras, and let $\displaystyle f=
\sum_{k=0}^\infty P_k$ be its Taylor series at zero, which is uniformly converging in $U=B_A (0,\delta)$. The following statements hold:\begin{enumerate}[$(a)$]
\item The mapping $f$ is orthogonally preserving on $U$ {\rm(}respectively, orthogonally preserving on $A_{sa}\cap U${\rm)} if, and only if, $P_0=0$ and the pair $(P_n,P_m)$ is orthogonality preserving {\rm(}respectively, orthogonally preserving on $A_{sa}${\rm)} for every $n,m\in \mathbb{N}$.
\item The mapping $f$ preserves zero products on $U$ if, and only if, $P_0 =0$ and for every $n,m\in \mathbb{N},$ the pair $(P_n,P_m)$ preserves zero products.
\end{enumerate}
\end{proposition}

\begin{proof} $(a)$ The ``if'' implication is clear. To prove the "only if" implication, let us fix $a,b\in U$ with $a\perp b$. Let us find two positive scalars $r,C$ such that $a,b\in B(0,r)$, and $\|f(x)\|\leq C$ for every $x\in B(0,r)\subset \overline{B}(0,r)\subseteq U$. From the Cauchy estimates we have $\|P_m \| \leq \frac{C}{r^{m}},$ for every $m\in \mathbb{N}\cup \{0\}.$
By hypothesis $f(t a)\perp f(t b)$, for every $r> t> 0$, and hence $$ P_0 (t a) P_0 (t b)^* +  P_0 (t a) \left(\sum_{k=1}^\infty P_k (t b) \right)^*+  \left(\sum_{k=1}^\infty P_k (t a) \right) \left(\sum_{k=0}^\infty P_k (t b)\right)^*=0,$$ and by homogeneity $$ P_0 (a) P_0 (b)^* = -  P_0 (a) \left(\sum_{k=1}^\infty t^{k} P_k (b) \right)^*+  \left(\sum_{k=1}^\infty t^{k} P_k (a) \right) \left(\sum_{k=0}^\infty t^k P_k (b)\right)^*.$$ Letting $t\to 0$, we have $P_0 (a) P_0 (b)^*=0$. In particular, $P_0 =0$.\smallskip

We shall prove by induction on $n$ that the pair $(P_j,P_k)$ is orthogonality preserving on $U$ for every $1\leq j,k \leq n$.
Since $f(t a) f(t b)^* =0$, we also deduce that
$$ P_1 (t a) P_1 (t b)^* +  P_1 (t a) \left(\sum_{k=2}^\infty P_k (t b) \right)^*+  \left(\sum_{k=2}^\infty P_k (t a) \right) \left(\sum_{k=1}^\infty P_k (t b)\right)^*=0,$$ for every $\frac{\min\{\|a\|,\|b\|\}}{r}>t>0,$ which implies that $$t^2 P_1 (a) P_1(b)^* = - t P_1 (a) \left(\sum_{k=2}^\infty t^k P_k (b) \right)^* - \left(\sum_{k=2}^\infty t^k P_k (a) \right) \left(\sum_{k=1}^\infty t^k P_k (b)\right)^*,$$ for every $\frac{\min\{\|a\|,\|b\|\}}{r}>t>0$, and hence
$$\left\| P_1 (a) P_1(b)^*  \right\| \leq t C \| P_1 (a)\|  \sum_{k=2}^\infty \frac{\|b\|^{k}}{r^{k}} t^{k-2} $$ $$+ t C^2 \left(\sum_{k=2}^\infty \frac{\|a\|^{k}}{r^{k}} t^{k-2} \right) \left(\sum_{k=1}^\infty \frac{\|b\|^{k}}{r^{k}} t^{k-1}\right).$$ Taking limit in $t\to 0$, we get $ P_1 (a) P_1(b)^*=0$. Let us assume that $(P_j,P_k)$ is orthogonality preserving on $U$ for every $1\leq j,k \leq n$. Following the argument above we deduce that $$P_1 (a) P_{n+1} (b)^* + P_{n+1} (a) P_{1} (b)^*  = - t P_1 (a) \left( \sum_{j=n+2}^{\infty} t^{j-n-2} P_j (b)\right)^* $$ $$- t \sum_{k=2}^{n} t^{k-2} P_k (a) \left(\sum_{j=n+1}^{\infty} t^{j-n-1} P_j (b) \right)^* - t P_{n+1} (a) \left( \sum_{j=2}^{\infty} t^{j-2} P_j (b)\right)^* $$ $$- t \left(\sum_{k=n+2}^{\infty} t^{k-n-2} P_k (a) \right) \left(\sum_{j=1}^{\infty} t^{j-1} P_j (b) \right)^*,$$ for every $\frac{\min\{\|a\|,\|b\|\}}{r}>|t|>0$. Taking limit in $t\to 0$, we have $$P_1 (a) P_{n+1} (b)^* + P_{n+1} (a) P_{1} (b)^*  =0.$$ Replacing $a$ with $s a$ ($s>0$) we get $$ s P_1 (a) P_{n+1} (b)^* + s^{n+1} P_{n+1} (a) P_{1} (b)^*  =0$$ for every $s>0$, which implies that $$P_1 (a) P_{n+1} (b)^*=0.$$ In a similar manner we prove that $P_k (a) P_{n+1} (b)^*=0$, for every $1\leq k\leq n+1$. The equalities $P_k (b)^* P_j (a)=0$ ($1\leq j,k\leq n+1$) follow similarly.\smallskip

We have shown that for each $n,m\in \mathbb{N}$, $P_n (a) \perp P_m (b)$ whenever $a,b\in U$ with $a\perp b$. Finally, taking $a,b\in A$ with $a\perp b$, we can find a positive $\rho$ such that $\rho a, \rho b\in U$ and $\rho a \perp \rho b$, which implies that $P_n (\rho a) \perp P_m (\rho b)$ for every $n,m\in \mathbb{N}$, witnessing that $(P_n,P_m)$ is orthogonality preserving for every $n,m\in \mathbb{N}$.\smallskip

The proof of $(b)$ follows in a similar manner.
\end{proof}

We can obtain now a corollary which is a first step toward the description of orthogonality preserving, orthogonally additive, holomorphic mappings between C$^*$-algebras.

\begin{corollary}\label{c OP + OA Holom} Let  $f:B_A (0,\varrho) \longrightarrow B$ be a holomorphic mapping, where $A$ and $B$ are C$^*$-algebras, and let $\displaystyle f=
\sum_{k=0}^\infty P_k$ be its Taylor series at zero, which is uniformly converging in $U=B_A (0,\delta)$. Suppose $f$ is orthogonality preserving on $A_{sa}\cap U$ and orthogonally additive {\rm(}respectively, orthogonally additive and zero products preserving{\rm)}. Then there exists a sequence $(T_n)$ of operators from $A$ into $B$ satisfying that the pair $(T_n,T_m)$ is orthogonality preserving on $A_{sa}$ {\rm(}respectively, zero products preserving on $A_{sa}${\rm)} for every $n,m\in \mathbb{N}$ and \begin{equation}\label{eq 1 OP OA Holom} f(x) = \sum_{n=1}^\infty T_n (x^n),
 \end{equation} uniformly in $x\in U$. In particular every $T_n$ is orthogonality preserving {\rm(}respectively, zero products preserving{\rm)} on $A_{sa}$. Furthermore, $f$ is symmetric if and only if every $T_n$ is symmetric.
\end{corollary}

\begin{proof} Combining Lemma \ref{l 1.1 in CarLassZal} and Proposition \ref{p OP holom mappings}, we deduce that $P_0=0$, $P_n$ is orthogonally additive and $(P_n,P_m)$ is orthogonality preserving on $A_{sa}$ for every $n,m$ in $\mathbb{N}.$ By Theorem \ref{thm PaPeVill}, for each natural $n$ there exists an operator $T_n : A \to B$ such that $\|P_n\| \leq \|T_n\| \leq 2 \|P_n\|$ and $$P_n (a) = T_n (a^n),$$ for every $a\in A$.\smallskip

Consider now two positive elements $a,b\in A$ with $a\perp b$ and fix $n,m\in \mathbb{N}$. In this case there exist positive elements $c,d$ in $A$ with $c^n =a$ and $d^m =b$ and $c\perp d$. Since the pair $(P_n,P_m)$ is orthogonality preserving on $A_{sa}$, we have $T_n (a) = T_n (c^n) =P_{n} (c)\perp P_m (d) = T_m (d^{m}) = T_m (b).$ Now, noticing that given $a,b$ in $A_{sa}$ with $a\perp b$, we can write $a= a^{+}-b^{-}$ and $b= b^{+}-b^{-}$, where $a^{\sigma},b^{\tau}$ are positive, $a^{+}\perp a^{-},$ $b^{+}\perp b^{-}$ and $a^{\sigma}\perp b^{\tau},$ for every $\sigma,\tau\in \{+,-\},$ we deduce that $T_n(a)\perp T_m(b)$. This shows that the pair $(T_n,T_m)$ is orthogonality preserving on $A_{sa}$.\smallskip

When $f$ orthogonally additive and zero products preserving the pair $(T_n,T_m)$ is zero products preserving on $A_{sa}$ for every $n,m\in \mathbb{N}$. The final statement is clear from Lemma \ref{l symmetric hol functions}.
\end{proof}

It should be remarked here that if a mapping $f:B_{A}(0,\delta) \longrightarrow B$ is given by an expression of the form in \eqref{eq 1 OP OA Holom} which uniformly converging in $U=B_{A}(0,\delta)$ where $(T_n)$ is a sequence of operators from $A$ into $B$ such that the pair $(T_n,T_m)$ is orthogonality preserving on $A_{sa}$ {\rm(}respectively, zero products preserving on $A_{sa}${\rm)} for every $n,m\in \mathbb{N}$, then $f$ is  orthogonally additive and orthogonality preserving on $A_{sa}\cap U$ {\rm(}respectively, orthogonally additive and zero products preserving{\rm)}.

\section{Orthogonality preserving pairs of operators}

Let $A$ and $B$ be two C$^*$-algebras. In this section we shall study those pairs of operators $S,T: A\to B$ satisfying that $S, T$ and the pair $(S,T)$ preserve orthogonality on $A_{sa}$. Our description generalizes some of the results obtained by M. Wolff in \cite{Wol} because a (symmetric) mapping $T: A\to B$ is orthogonality preserving on $A_{sa}$ if and only if the pair $(T,T)$ enjoys the same property. In particular, for every $^*$-homomorphism $\Phi : A \to B$, the pair $(\Phi,\Phi)$ preservers orthogonality. The same statement is true whenever $\Phi$ is a $^*$-anti-homomorphism, or a Jordan $^*$-homomorphism, or a triple homomorphism for the triple product $\J abc = \frac12 ( ab^*c + c b^* a)$.\smallskip

We observe that $S,T$ being symmetric implies that $(S,T)$ is orthogonality preserving on $A_{sa}$ if and only if $(S,T)$ is zero products preserving on $A_{sa}$. We shall offer here a newfangled and simplified proof which is also valid for pairs of operators.\smallskip

Let $a$ be an element in a von Neumann algebra $M$. We recall that the \emph{left} and \emph{right} \emph{support projections}
of $a$ (denoted by $l(a)$ and $d(a)$) are defined as follows: $l(a)$ (respectively, $d(a)$) is the smallest projection
$p\in M$ (respectively, $q\in M$) with the property that $pa = a$ (respectively, $aq=a$). It is known that when $a$ is hermitian $d(a) = l(a)$ is called the \emph{support} or \emph{range projection} of $a$ and is denoted by $s(a)$. It is also known that, for each $a=a^*$, the sequence $(a^{\frac{1}{3^n}})$ converges in the strong$^*$-topology of $M$ to $s(a)$ (cf. \cite[\S 1.10 and 1.11]{S}).\smallskip


An element $e$ in a C$^*$-algebra $A$ is said to be a \emph{partial isometry} whenever $e e^* e = e$ (equivalently, $e e^*$ or $e^* e$ is a projection in $A$). For each partial isometry $e$, the projections $e e^*$ and $e^* e$ are called the left and right support projections associated to $e$, respectively. Every partial isometry $e$ in $A$ defines a Jordan product and an involution on $A_e (e):=ee^* A e^*e$ given by $a \bullet_{_{e}} b =\frac12 ( a e^* b + b e^* a)$ and $a^{\sharp_{_e}} =  ea^* e$ ($a,b\in A_2 (e)$). It is known that $(A_2 (e), \bullet_{_{e}}, {\sharp_{_e}})$ is a unital JB$^*$-algebra with respect to its natural norm and $e$ is the unit element for the Jordan product $\bullet_{_{e}}$.\smallskip

Every element $a$ in a C$^*$-algebra $A$ admits a \emph{polar decomposition} in $A^{**}$, that is, $a$ decomposes uniquely as follows: $a = u |a|$, where $|a| = (a^*a) ^{\frac12}$ and $u$ is a partial isometry in $A^{**}$ such that $u^* u = s(|a|)$ and $u u^* = s(|a^*|)$ (compare \cite[Theorem 1.12.1]{S}). Observe that $u u^* a = a u^* u =u$. The unique partial isometry $u$ appearing in the polar decomposition of $a$ is called the range partial isometry of $a$ and is denoted by $r(a)$. Let us observe that taking $c= r(a) |a|^{\frac13}$, we have $c c^* c = a$. It is also easy to check that for each $b\in A$ with $b= r(a) r(a)^* b$ (respectively, $b= b r(a)^* r(a)$) the condition $a^* b =0$ (respectively, $b a^*=0$) implies $b=0$. Furthermore, $a\perp b$ in $A$ if and only if $r(a) \perp r(b)$ in $A^{**}$.\smallskip

We begin with a basic argument in the study of orthogonality preserving operators between C$^*$-algebras whose proof is inserted here for completeness reasons. Let us recall that for every C$^*$-algebra $A$, the \emph{multiplier algebra} of $A$, $M(A)$, is the set of all elements $x\in A^{**}$ such that for each $Ax, xA \subseteq A$. We notice that $M(A)$ is a C$^*$-algebra and contains the unit element of $A^{**}$.

\begin{lemma}\label{l multiplier} Let $A$ and $B$ be C$^*$-algebras and let $S,T:A \to B$ be a pair of operators.\begin{enumerate}[$(a)$]
\item  The pair $(S,T)$ preserves orthogonality {\rm(}on $A_{sa}${\rm)} if and only if the pair $(S^{**}|_{M(A)},T^{**}|_{M(A)})$ preserves orthogonality {\rm(}on $M(A)_{sa}${\rm)};
\item  The pair $(S,T)$ preserves zero products {\rm(}on $A_{sa}${\rm)} if and only if the pair $(S^{**}|_{M(A)},T^{**}|_{M(A)})$ preserves zero products {\rm(}on $M(A)_{sa}${\rm)}.
\end{enumerate}
\end{lemma}

\begin{proof} $(a)$ The ``if'' implication is clear. Let $a,b$ be two elements in $M(A)$ with $a\perp b$. We can find two elements $c$ and $d$ in $M(A)$ satisfying $c c^* c =a$, $d d^* d =b$ and $c\perp d$. Since $c x c \perp dy d$, for every $x,y$ in $A$, we have $T(cxc) \perp T(dyd)$ for every $x,y\in A$.
By Goldstine's theorem we find two bounded nets $(x_{\lambda})$ and $(y_{\mu})$ in $A$,
converging in the weak$^*$ topology of $A^{**}$ to $c^*$ and $d^*$, respectively. Since $T(c x_{\lambda} c) T(dy_{\mu} d)^* = T(d y_{\mu} d)^* T(cx_{\lambda} c) = 0$, for every $\lambda, \mu$, $T^{**}$ is weak$^*$-continuous, the product of $A^{**}$ is separately weak$^*$-continuous and the involution of $A^{**}$ also is weak$^*$-continuous, we get $T^{**} (cc^*c) T^{**} (dd^* d) = T^{**} (a) T^{**} (b)^* =0= T^{**} (b)^* T^{**} (a),$ and hence $T^{**} (a) \perp T^{**} (b)$, as desired.\smallskip

The proof of $(b)$ follows by a similar argument.
\end{proof}

\begin{proposition}\label{p symmetric OP pairs} Let $S,T: A\to B$ be operators between C$^*$-algebras such that $(S,T)$ is orthogonality preserving on $A_{sa}$. Let us denote $h:= S^{**} (1)$ and $k:= T^{**} (1)$. Then the identities $$S(a) T(a^*)^* = S(a^2) k^*= h T((a^2)^*)^*,$$
$$ T(a^*)^* S(a) =k^* S(a^2)= h T((a^2)^*)^* h,$$
$$S(a) k^* = h T(a^*)^*, \hbox{ and, } k^* S(a) = T(a^*)^* h$$ hold for every $a\in A$.
\end{proposition}

\begin{proof} By Lemma \ref{l multiplier}, we may assume, without loss of generality, that $A$ is unital.
$(a)$ For each $\varphi \in B^*$, the continuous bilinear form $V_{\varphi} : A \times A \to \mathbb{C}$, $V_{\varphi} (a,b)= \varphi (S(a) T(b^*)^*)$ is orthogonal, that is, $V_{\varphi} (a,b) = 0$, whenever $a b=0 $ in $A_{sa}$. By Goldstein's theorem \cite[Theorem 1.10]{Gold} there exist functionals $\omega_1,\omega_2\in A^*$ satisfying that $$V_{\varphi} (a,b) = \omega_1 (a b) + \omega_2 (b a),$$ for all $a,b\in A$. Taking $b=1$ and $a=b$ we have $$\varphi (S(a) k^*) = V_{\varphi} (a,1)= V_{\varphi} (1,a) = \varphi (h T(a)^*)$$ and $$\varphi (S(a) T(a)^*) = \varphi (S(a^2) k^*)=  \varphi (h T(a^2)^*),$$ for every $a\in A_{sa}$, respectively. Since $\varphi$ was arbitrarily chosen, we get, by linearity, $S(a) k^* = h T(a^*)^*$ and $S(a) T(a^*)^* = S(a^2) k^*= h T((a^2)^*)^*$, for every $a\in A$. The other identities follow in a similar way, but replacing $V_{\varphi} (a,b)= \varphi (S(a) T(b^*)^*)$ with $V_{\varphi} (a,b)= \varphi ( T(b^*)^* S(a))$.
\end{proof}

\begin{lemma}\label{c OP pairs Jordan *-homomorphisms}
Let $J_1,J_2: A\to B$ be Jordan $^*$-homomorphism between C$^*$-algebras. The following statements are equivalent: \begin{enumerate}[$(a)$]
\item The pair $(J_1,J_2)$ is orthogonality preserving on $A_{sa}$;
\item The identity $$ J_1 (a) J_2 (a) = J_1 (a^2) J_2^{**} (1) = J_1^{**} (1) J_2 (a^2),$$ holds for every $a\in A_{sa}$;
\item The identity $$ J_1^{**} (1) J_2 (a) = J_1 (a) J_2^{**} (1),$$ holds for every $a\in A_{sa}$.
\end{enumerate} Furthermore, when $J_1^{**}$ is unital, $J_2 (a) = J_1 (a) J_2^{**} (1)= J_2^{**} (1)  J_1 (a),$ for every $a$ in $A.$
\end{lemma}

\begin{proof} The implications $(a)\Rightarrow (b)\Rightarrow (c)$ have been established in Proposition \ref{p symmetric OP pairs}. To see $(c) \Rightarrow (a)$, we observe that $J_i (x) =J_i^{**} (1) J_i (x) J_i^{**} (1)= J_i (x) J_i^{**} (1) = J_i^{**} (1) J_i (x)$, for every $x\in A$. Therefore, given $a,b\in A_{sa}$ with $a\perp b$, we have $J_1 (a) J_2 (b) =  J_1 (a) J_1^{**} (1) J_2 (b) = J_1 (a) J_1 (b) J_2^{**} (1) =0$.
\end{proof}

In \cite[Proposition 2.5]{Wol}, M. Wolff establishes a uniqueness result for
$^*$-homomorphisms between C$^*$-algebras showing that for each pair $(U,V)$ of unital $^*$-homomorphisms from a unital C$^*$-algebra $A$
into a unital C$^*$-algebra $B$, the condition $(U,V)$ orthogonality preserving on $A_{sa}$ implies $U=V$. This uniqueness result is a direct consequence of our previous lemma.\smallskip

Orthogonality preserving pairs of operators can be also used to rediscover the notion of orthomorphism in the sense introduced by Zaanen in \cite{Za}. We recall that an operator $T$ on a C$^*$-algebra $A$ is said to be an \emph{orthomorphism} or a \emph{band preserving} operator when the implication $a\perp b \Rightarrow T(a) \perp b$ holds for every $a,b\in A$. We notice that when $A$ is regarded as an $A$-bimodule,
an operator $T:A\to A$ is an orthomorphism if and only if it is a \emph{local operator} in the sense used by B.E. Johnson in \cite[\S 3]{John01}. Clearly, an operator $T:A\to A$ is an orthomorphism if and only if $(T,Id_{A})$ is orthogonality preserving. The following non-commutative extension of \cite[THEOREM 5]{Za} follows from Proposition \ref{p symmetric OP pairs}.

\begin{corollary}\label{c non-commutatice Zaanen} Let $T$ be an operator on a C$^*$-algebra $A$. Then $T$ is an orthomorphism if and only if $T(a) = T^{**} (1) a = a T^{**} (1)$, for every $a$ in $A$, that is, $T$ is a multiple of the identity on $A$ by an element in its center.$\hfill\Box$
\end{corollary}

We recall that two elements $a,$ $b$ in a JB$^*$-algebra $A$ are said to \emph{operator commute} in $A$ if the
multiplication operators $M_a$ and $M_b$ commute, where $M_a$ is defined by $M_a (x) := a\circ x$. That is, $a$ and $b$ operator commute if and only if $(a\circ x) \circ b = a\circ (x\circ b)$ for all $x$ in $A$. An useful result in Jordan theory assures that self-adjoint elements $a$ and $b$ in $A$ generate a JB$^*$-subalgebra that can be realized as a JC$^*$-subalgebra of some $B(H)$ (compare \cite{Wri77}), and, under this identification, $a$ and $b$ commute as elements in $L(H)$ whenever they operator commute in $A$, equivalently $a^2 \circ b = 2 (a\circ b)\circ a - a^2 \circ b$ (cf. Proposition 1 in \cite{Top}).\smallskip

The next lemma contains a property which is probably known in C$^*$-algebra, we include an sketch of the proof because we were unable to find an explicit reference.

\begin{lemma}\label{l operator comm Peirce 2} Let $e$ be a partial isometry in a C$^*$-algebra $A$ and let $a,b$ be two elements in $A_2 (e) = e e^* A e^*e$. Then $a$, $b$ operator commute in the JB$^*$-algebra $(A_2 (e), \bullet_{_{e}}, {\sharp_{_e}})$ if and only if $ae^*$ and $be^*$ operator commute in the JB$^*$-algebra $(A_2 (ee^*), \bullet_{_{ee^*}}, {\sharp_{_{ee^*}}})$, where $x \bullet_{_{ee^*}} y = x\circ y = \frac12 (x y +yx)$, for every $x,y\in A_2 (ee^*)$. Furthermore, when $a$ and $b$ are hermitian elements in $(A_2 (e), \bullet_{_{e}}, {\sharp_{_e}})$, $a$, $b$ operator commute if and only if $ae^*$ and $b e^*$ commute in the usual sense {\rm(}i.e. $ae^*be^*=be^* ae^*${\rm)}.
\end{lemma}

\begin{proof} We observe that the mapping $R_{e^*} : (A_2 (e), \bullet_{_{e}}) \to (A_2 (ee^*),\bullet_{_{ee^*}})$, $x\mapsto x e^*$ is a Jordan $^*$-isomorphism between the above JB$^*$-algebras. So, the first equivalence is clear. The second one has been commented before.
\end{proof}

Our next corollary relies on the following description of orthogonality preserving operators between C$^*$-algebras obtained in \cite{BurFerGarMarPe} (see also \cite{BurFerGarPe}).

\begin{theorem}\label{thm BurFerGarMarPe}\cite[Theorem 17]{BurFerGarMarPe}, \cite[Theorem 4.1 and Corollary 4.2]{BurFerGarPe} Let $T$ be an operator from a C$^*$-algebra $A$ into another C$^*$-algebra $B$ the following are equivalent: \begin{enumerate}[{\rm $a)$}]
\item $T$ is orthogonality preserving (on $A_{sa}$). \item There exits a unital Jordan $^*$-homomorphism $J: M(A) \to B_{2}^{**} (r(h))$
such that $J(x)$ and $h=T^{**} (1)$ operator commute and $$T(x) =
h\bullet_{_{r(h)}} J(x), \hbox{ for every $x\in A$},$$ where $M(A)$ is the multiplier algebra of $A$, $r(h)$ is the range partial isometry of $h$ in $B^{**}$, $B_{2}^{**} (r(h)) = r(h) r(h)^* B^{**} r(h)^* r(h)$ and $\bullet_{_{r(h)}}$ is the natural product making $B_{2}^{**} (r(h))$ a JB$^*$-algebra.
\end{enumerate} Furthermore, when $T$ is symmetric, $h$ is hermitian and hence $r(h)$ decomposes as orthogonal sum of two projections in $B^{**}$.$\hfill\Box$
\end{theorem}

Our next result gives a new perspective for the study of orthogonality preserving (pairs of) operators between C$^*$-algebras.

\begin{proposition}\label{c pairs OP and OP operators}
Let $A$ and $B$ be C$^*$-algebras. Let $S,T: A\to B$ be operators and let $h = S^{**} (1)$ and $k= T^{**} (1)$. Then the following statements hold:\begin{enumerate}
[$(a)$] \item The operator $S$ is orthogonality preserving if and only if there exit two Jordan $^*$-homomorphisms $\Phi,\widetilde{\Phi}: M(A) \to B^{**}$ satisfying $\Phi (1) = r(h) r(h)^*$, $\widetilde{\Phi} (1) = r(h)^* r(h),$ and $S(a) = \Phi (a) h = h \widetilde{\Phi} (a),$ for every $a\in A$.
\item $S,T$ and $(S,T)$ are orthogonality preserving on $A_{sa}$ if and only if the following statements hold:
\begin{enumerate}[$(b1)$]
\item There exit Jordan $^*$-homomorphisms $\Phi_1,\widetilde{\Phi}_1, \Phi_2,\widetilde{\Phi}_2: M(A) \to B^{**}$ satisfying $\Phi_1 (1) = r(h) r(h)^*$, $\widetilde{\Phi}_1 (1) = r(h)^* r(h),$ $\Phi_2 (1) = r(k) r(k)^*$, $\widetilde{\Phi}_2 (1) = r(k)^* r(k),$ $S(a) = \Phi_1 (a) h = h \widetilde{\Phi}_1 (a),$ and $T(a) = \Phi_2 (a) k = k \widetilde{\Phi}_2 (a),$ for every $a\in A$;
\item The pairs $( {\Phi}_1 ,  {\Phi}_2)$ and $(\widetilde{\Phi}_1, \widetilde{\Phi}_2)$ are orthogonality preserving on $A_{sa}$.
\end{enumerate}
\end{enumerate}
\end{proposition}

\begin{proof} The ``if'' implications are clear in both statements. We shall only prove the ``only if'' implication.\smallskip

$(a)$. By Theorem \ref{thm BurFerGarMarPe}, there exits a unital Jordan $^*$-homomorphism $J_1: M(A) \to B_{2}^{**} (r(h))$ such that $J_1(x)$ and $h$ operator commute in the JB$^*$-algebra $(B_{2}^{**} (r(h)), \bullet_{_{r(h)}})$ and $$S(x) = h\bullet_{_{r(a)}} J_1(a) \hbox{ for every $a\in A$}.$$

Fix $a\in A_{sa}$. Since $h$ and $J_1 (a)$ are hermitian elements in $(B_{2}^{**} (r(h)), \bullet_{_{r(h)}})$ which operator commute, Lemma \ref{l operator comm Peirce 2} assures that $h r(h)^*$ and $J_1(a) r(h)^*$ commute in the usual sense of $B^{**}$, that is, $$h r(h)^* J_1(a) r(h)^* = J_1(a) r(h)^*h r(h)^*,$$ or equivalently, $$h r(h)^* J_1(a) = J_1(a) r(h)^*h.$$ Consequently, we have $$S(a) = h \bullet_{_{r(h)}} J_1 (a) = h r(h)^* J_1 (a) = J_1 (a) r(h)^* h,$$ for every $a\in A$. The desired statement follows by considering $\Phi_1 (a) = J_1 (a) r(h)^*$ and $\widetilde{\Phi}_1 (a) =r(h)^* J_1 (a).$\smallskip

$(b)$ The statement in $(b1)$ follows from $(a)$. We shall prove $(b2)$.\smallskip

By hypothesis, given $a,b$ in $A_{sa}$ with $a\perp b$, we have $$0= S(a) T(b)^*= \left(h \widetilde{\Phi}_1 (a) \right) \left(k \widetilde{\Phi}_2 (b)\right)^*$$ $$=  h \widetilde{\Phi}_1 (a) \widetilde{\Phi}_2 (b)^* k^* $$ Having in mind that $\widetilde{\Phi}_1 (A)\subseteq r(h)^* r(h) B^{**}$ and $\widetilde{\Phi}_2 (A)\subseteq  B^{**} r(k)^* r(k)$, we deduce that $\widetilde{\Phi}_1 (a) \widetilde{\Phi}_2 (b)^* =0$ (compare the comments before Lemma \ref{l multiplier}), as we desired. In a similar fashion we prove $\widetilde{\Phi}_2 (b)^* \widetilde{\Phi}_1 (a)=0$, ${\Phi}_2 (b)^* {\Phi}_1 (a)=0 = {\Phi}_1 (a)  {\Phi}_2 (b)^*.$
\end{proof}

\section{Holomorphic mappings valued in a commutative C$^*$-algebra}

The particular setting in which a holomorphic function is valued in a commutative C$^*$-algebra provides enough advantages to establish a full description of the orthogonally additive, orthogonality preserving, holomorphic mappings which are valued in a commutatively C$^*$-algebra.

\begin{proposition}\label{p 1 comm} Let $S,T: A \to B$ be operators between C$^*$-algebras with $B$ commutative. Suppose that $S$, $T$ and $(S,T)$ are orthogonality preserving, and let us denote $h = S^{**} (1)$ and $k= T^{**} (1)$. Then there exits a Jordan $^*$-homomorphism $\Phi: M(A) \to B^{**}$ satisfying $\Phi (1) = r(|h| +|k|)$, $S(a) = \Phi (a) h,$ and $T(a) = \Phi (a) k,$ for every $a\in A$.
\end{proposition}

\begin{proof} Let $\Phi_1,\widetilde{\Phi}_1, \Phi_2,\widetilde{\Phi}_2: M(A) \to B^{**}$ be the Jordan $^*$-homomorphisms satisfying $(b1)$ and $(b2)$ in Proposition \ref{c pairs OP and OP operators}. By hypothesis, $B$ is commutative, and hence $\Phi_i=\widetilde{\Phi}_i$ for every $i=1,2$ (compare the proof of Proposition \ref{c pairs OP and OP operators}). Since the pair $( {\Phi}_1 ,  {\Phi}_2)$ is orthogonality preserving on $A_{sa}$, Lemma \ref{c OP pairs Jordan *-homomorphisms} assures that $$ {\Phi}_1^{**} (1) {\Phi}_2 (a) = {\Phi}_1 (a) {\Phi}_2^{**} (1),$$ for every $a\in A_{sa}$. In order to simplify notation, let us denote $p= {\Phi}_1^{**} (1)$ and $q = {\Phi}_2^{**} (1)$.\smallskip

We define an operator ${\Phi}: M(A) \to B^{**}$, defined by $$\Phi (a) = pq \Phi_1 (a) + p(1-q) \Phi_1 (a) + q(1-p) \Phi_2 (a).$$ Since $p {\Phi}_2 (a) = {\Phi}_1 (a) q$, it can be easily checked that $\Phi$ is a Jordan $^*$-homomorphism such that $S(a) = \Phi (a) h,$ and $T(a) = \Phi (a) k,$ for every $a\in A$.
\end{proof}

\begin{theorem}\label{thm OP + OA Holom commutative} Let  $f:B_A (0,\varrho) \longrightarrow B$ be a holomorphic mapping, where $A$ and $B$ are C$^*$-algebras with $B$ commutative, and let $\displaystyle f=
\sum_{k=0}^\infty P_k$ be its Taylor series at zero, which is uniformly converging in $U=B_A (0,\delta)$. Suppose $f$ is orthogonality preserving on $A_{sa}\cap U$ and orthogonally additive {\rm(}equivalently, orthogonally additive and zero products preserving{\rm)}. Then there exist a sequence $(h_n)$ in $B^{**}$ and a Jordan $^*$-homomorphism $\Phi : M(A) \to B^{**}$ such that $$ f(x) = \sum_{n=1}^\infty h_n \Phi (a^n)= \sum_{n=1}^\infty h_n \Phi (a^n),$$ uniformly in $a\in U$.
\end{theorem}

\begin{proof} By Corollary \ref{c OP + OA Holom}, there exists a sequence $(T_n)$ of operators from $A$ into $B$ satisfying that the pair $(T_n,T_m)$ is orthogonality preserving on $A_{sa}$ {\rm(}equivalently, zero products preserving on $A_{sa}${\rm)} for every $n,m\in \mathbb{N}$ and $$ f(x) = \sum_{n=1}^\infty T_n (x^n),$$ uniformly in $x\in U$. Denote $h_n= T_n^{**} (1)$.\smallskip

We shall prove now the existence of the Jordan $^*$-homomorphism $\Phi$. We prove, by induction, that for each natural $n$, there exists a Jordan $^*$-homomorphism $\Psi_n : M(A) \to B^{**}$ such that $r(\Psi_n (1))= r(|h_1|+\ldots+|h_n|)$ and $T_k (a) = h_k \Psi_n (a)$ for every $k\leq n$, $a\in A$. The statement for $n=1$ follows from Corollary \ref{c OP + OA Holom} and Proposition \ref{c pairs OP and OP operators}. Let us assume that our statement is true for $n$. Since for every $k,m$ in $\mathbb{N}$, $T_k$, $T_m$ and the pair $(T_k,T_m)$ are orthogonality preserving, we can easily check that $T_{n+1}$, $T_1 +\ldots+ T_n$ and $(T_{n+1},T_1 +\ldots+ T_n)= (T_{n+1}, (h_1 +\ldots+ h_n) \Psi_n)$ are  orthogonality preserving. By Proposition \ref{p 1 comm}, there exists a Jordan $^*$-homomorphism $\Psi_{n+1}: M(A) \to B^{**}$ satisfying $r(\Psi_{n+1} (1)) = r (|h_1|+\ldots+|h_n|+|h_{n+1}|)$, $T_{n+1} (a) = h_{n+1} \Psi_{n+1} (a^{n+1})$ and $(T_1 +\ldots+ T_n)(a) = (h_1 +\ldots +h_n) \Psi_{n+1} (a) $ for every $k\leq n$, $a\in A$. Since for each, $1\leq k\leq n$, $$h_k \Psi_{n+1} (a) = h_{k} r (|h_1|+\ldots+|h_n|+|h_{n+1}|) \Psi_{n+1} (a) $$ $$= h_k (|h_1|+\ldots+|h_n|) \Psi_{n+1} (a)$$ $$= h_k (|h_1|+\ldots+|h_n|) \Psi_{n} (a)= h_k \Psi_{n} = T_k (a),$$ for every $a\in A$, as desired.\smallskip

Let us consider a free ultrafilter $\mathcal{U}$ on $\NN$. By the Banach-Alaoglu theorem, any bounded set in $B^{**}$ is relatively weak$^*$-compact, and thus the assignment $a\mapsto \Phi(a) := w^*-\lim_{\mathcal{U}} \Psi_{n} (a)$ defines a Jordan $^*$-homomorphism from $M(A)$ into $B^{**}$. If we fix a natural $k$, we know that $T_k (a) = h_k  \Psi_n (a)$, for every $n\geq k$ and $a\in A$. Then it can be easily checked that $T_k (a) = h_k \Phi (a),$ for every $a\in A$, which concludes the proof.
\end{proof}

The Banach-Stone type theorem for orthogonally additive, orthogonality preserving, holomorphic mappings between commutative C$^*$-algebras, established in Theorem \ref{t BuShuWong} (see \cite[Theorem 3.4]{BuHsuWong2013}) is a direct consequence of our previous result.


\section{Banach-Stone type theorems for holomorphic mappings between general C$^*$-algebras}

In this section we deal with holomorphic functions between general C$^*$-algebras. In this more general setting we shall require additional hypothesis to establish a result in the line of the above Theorem \ref{thm OP + OA Holom commutative}.\smallskip

Given a unital C$^*$-algebra $A$, the symbol inv$(A)$ will denote the set of invertible elements in $A$. The next lemma is a technical tool which is needed later. The proof is left to the reader and follows easily from the fact that inv$(A)$ is an open subset of $A$.

\begin{lemma}\label{l holom with invertible in its image} Let  $f:B_A (0,\varrho) \longrightarrow B$ be a holomorphic mapping, where $A$ and $B$ are C$^*$-algebras with $B$ unital, and let $\displaystyle f=
\sum_{k=0}^\infty P_k$ be its Taylor series at zero, which is uniformly converging in $U=B_A (0,\delta)$. Let us assume that there exists $a_0\in U$ with $f(a_0)\in \hbox{inv} (B)$. Then there exists $m_0\in \mathbb{N}$ such that $\displaystyle\sum_{k=0}^{m_0} P_k (a_0)\in \hbox{inv} (B)$.$\hfill\Box$
\end{lemma}

We can now state a description of those orthogonally additive, orthogonality preserving, holomorphic mappings between C$^*$-algebras whose image contains an invertible element.

\begin{theorem}\label{thm OP + OA Holom general} Let  $f:B_A (0,\varrho) \longrightarrow B$ be a holomorphic mapping, where $A$ and $B$ are C$^*$-algebras with $B$ unital, and let $\displaystyle f=
\sum_{k=0}^\infty P_k$ be its Taylor series at zero, which is uniformly converging in $U=B_A (0,\delta)$.
Suppose $f$ is orthogonality preserving on $A_{sa}\cap U$, orthogonally additive on $U$ and $f(U)\cap \hbox{inv} (B)\neq \emptyset$. Then there exist a sequence $(h_n)$ in $B^{**}$ and Jordan $^*$-homomorphisms $\Theta, \widetilde{\Theta} : M(A) \to B^{**}$ such that $$ f(a) = \sum_{n=1}^\infty h_n \widetilde{\Theta} (a^n)= \sum_{n=1}^\infty {\Theta} (a^n) h_n,$$ uniformly in $a\in U$.
\end{theorem}

\begin{proof} By Corollary \ref{c OP + OA Holom} there exists a sequence $(T_n)$ of operators from $A$ into $B$ satisfying that the pair $(T_n,T_m)$ is orthogonality preserving on $A_{sa}$ for every $n,m\in \mathbb{N}$ and $$ f(x) = \sum_{n=1}^\infty T_n (x^n),$$ uniformly in $x\in U$.\smallskip

Now,  Proposition \ref{c pairs OP and OP operators} $(a)$, applied to $T_n$ ($n\in \mathbb{N}$), implies the existence of sequences $(\Phi_n)$ and $(\widetilde{\Phi}_n)$ of Jordan $^*$-homomorphisms from $M(A)$ into $B^{**}$ satisfying $\Phi_n (1) = r(h_n) r(h_n )^*$, $\widetilde{\Phi}_n (1) = r(h_n)^* r(h_n),$ where $h_n = T_n^{**} (1)$, and $$T_n (a) = \Phi_n (a) h_n = h_n \widetilde{\Phi}_n (a),$$ for every $a\in A$, $n\in \mathbb{N}$. Moreover, from Proposition \ref{c pairs OP and OP operators} $(b)$, the pairs $( {\Phi}_n ,  {\Phi}_m)$ and $(\widetilde{\Phi}_n, \widetilde{\Phi}_m)$ are orthogonality preserving on $A_{sa}$, for every $n,m\in \mathbb{N}$.\smallskip

Since $f(U)\cap \hbox{inv} (B)\neq \emptyset$, it follows from Lemma \ref{l holom with invertible in its image} that there exists a natural $m_0$ and $a_0\in A$ such that $$\displaystyle\sum_{k=1}^{m_0} P_k (a_0) = \sum_{k=1}^{m_0} \Phi_k (a_0^k) h_k = \sum_{k=1}^{m_0} h_k \widetilde{\Phi}_k (a_0^k) \in \hbox{inv} (B).$$

We claim that $r(h_1)^* r(h_1) +\ldots + r(h_{m_0})^* r(h_{m_0})$ is invertible in $B^{+}$ (and in $B^{**}$). Otherwise, we could find a projection $q\in B^{**}$ satisfying $(r(h_1)^* r(h_1) +\ldots + r(h_{m_0})^* r(h_{m_0})) q = 0$. This would imply that $$\left( \sum_{k=1}^{m_0} P_k (a_0)\right) q = \left(\sum_{k=1}^{m_0} \Phi_k (a_0^k) h_k\right) q = 0,$$ contradicting that $\displaystyle\sum_{k=1}^{m_0} P_k (a_0) = \sum_{k=1}^{m_0} \Phi_k (a_0^k) h_k$ is invertible in $B$.\smallskip

Consider now the mapping $\Psi = \sum_{k=1}^{m_0} \widetilde{\Phi}_k$. It is clear that, for each natural $n$, $\Psi$, $\widetilde{\Phi}_n$ and the pair $(\Psi, \widetilde{\Phi}_n)$ are orthogonality preserving. Applying Proposition \ref{c pairs OP and OP operators} $(b)$, we deduce the existence of Jordan $^*$-homomorphisms $\Theta, \widetilde{\Theta}, \Theta_n, \widetilde{\Theta}_n: M(A) \to B^{**}$ such that $(\Theta,\Theta_n)$ and $(\widetilde{\Theta},\widetilde{\Theta}_n)$ are orthogonality preserving, $\Theta (1) = r(k) r(k)^*$, $\widetilde{\Theta} (1) = r(k)^* r(k)$, $\Theta_n (1) = r(h_n) r(h_n)^*$, $\widetilde{\Theta}_n (1) = r(h_n)^* r(h_n)$, $$\Psi (a) = \Theta (a) k = k \widetilde{\Theta} (a)$$ and $$\widetilde{\Phi}_n (a) = \Theta_n (a) r(h_n)^* r(h_n) = r(h_n)^* r(h_n) \widetilde{\Theta}_n (a),$$ for every $a\in A$, where $k = \Psi (1) = r(h_1)^* r(h_1) +\ldots + r(h_{m_0})^* r(h_{m_0})$. The invertibility of $k$, proved in the previous paragraph, shows that $\Theta (1) = 1$. Thus, since $(\widetilde{\Theta},\widetilde{\Theta}_n)$ is orthogonality preserving, the last statement in Lemma \ref{c OP pairs Jordan *-homomorphisms} proves that $$\widetilde{\Theta}_n (a)= \widetilde{\Theta}_n (1) \widetilde{\Theta} (a)=  \widetilde{\Theta} (a)  \widetilde{\Theta}_n (1),$$ for every $a\in A$, $n\in \mathbb{N}$. The above identities guarantee that $$\widetilde{\Phi}_n (a) = \Theta (a) r(h_n)^* r(h_n) = r(h_n)^* r(h_n) \widetilde{\Theta} (a),$$ for every $a\in A$, $n\in \mathbb{N}$.\smallskip

A similar argument to the one given above, but replacing $\widetilde{\Phi}_k$ with ${\Phi}_k$, shows the existence of a Jordan $^*$-homomorphism $\Theta : M(A) \to B^{**}$ such that  $${\Phi}_n (a) = \Theta (a) r(h_n) r(h_n)^* = r(h_n) r(h_n)^* {\Theta} (a),$$ for every $a\in A$, $n\in \mathbb{N}$, which concludes the proof.
\end{proof}

\bibliographystyle{amsplain}

\end{document}